\documentclass[10pt]{amsart}
\usepackage{setspace}

\usepackage{hyperref}

\usepackage{todonotes}
\usepackage{multirow}
\usepackage{longtable}

\usepackage{amsmath}
\usepackage{amsfonts}
\usepackage{amssymb}
\usepackage{float}
\usepackage{amsthm}
\usepackage{comment}
\usepackage{amscd}
\usepackage{amsxtra}
\usepackage{epsfig}

\usepackage{listings}
\usepackage{color}

\definecolor{dkgreen}{rgb}{0,0.6,0}
\definecolor{gray}{rgb}{0.5,0.5,0.5}
\definecolor{mauve}{rgb}{0.58,0,0.82}

\usepackage{color}
\usepackage[all]{xy}
\usepackage{verbatim}
\usepackage{setspace}
\usepackage{eucal}
\usepackage{graphicx}
\usepackage{tikz-cd}

\usepackage{url}
\usepackage{verbatim}

\usepackage{xy}
\xyoption{all}
\newcommand{\xar}[1]{\xrightarrow{{#1}}}

\DeclareMathOperator{\Hom}{Hom}

\DeclareMathOperator{\Ext}{Ext}
\DeclareMathOperator{\spec}{Spec}
\DeclareMathOperator{\spf}{Spf}

\DeclareMathOperator{\Gal}{Gal}

\DeclareMathOperator{\coker}{coker}

\DeclareMathOperator{\QCoh}{\mathrm{QCoh}}

\DeclareMathOperator{\pic}{Pic}
\DeclareMathOperator{\colim}{colim}

\DeclareMathOperator{\Tot}{Tot}

\newtheorem{lemma}{Lemma}[section]

\newtheorem{corollary}[lemma]{Corollary}

\newtheorem{theorem}[lemma]{Theorem}

\newtheorem{prop}[lemma]{Proposition}

\theoremstyle{definition}

\newtheorem{conjecture}[lemma]{Conjecture}
\newtheorem{definition}[lemma]{Definition}
\newtheorem{example}[lemma]{Example}

\newtheorem{remark}[lemma]{Remark}

\newcommand{\FF}{\mathbf{F}}
\newcommand{\Z}{\mathbf{Z}}
\newcommand{\QQ}{\mathbf{Q}}
\newcommand{\cc}{\mathcal{C}}

\newcommand{\Lp}{L_{K(p-1)}}
\newcommand{\Lk}{L_{K(n)}}

\newcommand{\Sp}{\mathrm{Sp}}

\newcommand{\Tmf}{\mathrm{Tmf}}

\newcommand{\Aff}[1]{{\mathrm{Aff}^{\text{\'{e}t}}_{/#1}}}

\renewcommand{\H}{\mathrm{H}}
\newcommand{\M}{\mathcal{M}}

\newcommand{\Mfg}{\M_{FG}}

\newcommand{\ff}{\mathcal{F}}
\newcommand{\co}{\mathcal{O}}

\renewcommand{\SS}{\mathbf{S}}
\newcommand{\GG}{\mathbf{G}}

\newcommand{\Mod}{\mathrm{Mod}}
\newcommand{\cL}{\mathcal{L}}

\newcommand{\CP}{\mathbf{C}P}

\newcommand{\Map}{\mathrm{Map}}

\newcommand{\dZ}{\mathfrak{Z}}
\newcommand{\dX}{\mathfrak{X}}
\newcommand{\frakm}{\mathfrak{m}}
\newcommand{\Lone}{L_{K(1)}}
\newcommand{\Ltwo}{L_{K(2)}}

\newcommand{\cg}{\mathcal{G}}

\newcommand{\der}{\mathrm{der}}

\newcommand{\Eoo}{{\mathbf{E}_\infty}}
\newcommand{\EOo}{{\mathbf{E}_\infty}}

\newcommand{\dY}{\mathfrak{Y}}

\newcommand{\cf}{\mathcal{F}}
\newcommand{\DD}{\mathbf{D}}

\newcommand{\ev}{\mathrm{ev}}

\newcommand{\bgl}{\mathrm{bgl}}

\newcommand{\GL}{\mathrm{GL}}

\newcommand{\wt}[1]{\widetilde{#1}}
\newcommand{\fr}[1]{\mathfrak{#1}}

\title{The Lubin-Tate stack and Gross-Hopkins duality}
\author{Sanath Devalapurkar}
\begin{document}

\begin{abstract}
    Morava $E$-theory $E$ is an $\Eoo$-ring with an action of the Morava
    stabilizer group $\Gamma$. We study the derived stack $\spf E/\Gamma$.
    Descent-theoretic techniques allow us to deduce a theorem of
    Hopkins-Mahowald-Sadofsky on the $K(n)$-local Picard group. These
    techniques also allow us to rederive a few consequences of a recent result
    of Barthel-Beaudry-Stojanoska on the Anderson duals of higher real
    $K$-theories.
\end{abstract}
\maketitle

\section{Introduction}
Goerss-Hopkins-Miller proved that Morava $E$-theory $E$ (at a fixed height $n$
and prime $p$) is an $\Eoo$-ring. Moreover, the profinite group $\Gamma$ (also
known as the Morava stabilizer group) of units in a certain division algebra of
Hasse invariant $1/n$ acts continuously on $E$ via $\Eoo$-ring maps. From the
perspective of derived algebraic geometry, this is saying that one can
construct the object $\spf E/\Gamma$ (the ``Lubin-Tate stack'').

Devinatz and Hopkins proved that there is an equivalence
$$\Lk S \simeq E^{h\Gamma},$$
where the right hand side uses an appropriate notion of continuous fixed
points. This result allows us to show that there is an equivalence of
$\infty$-categories
$$\QCoh(\spf E/\Gamma) \simeq \Lk \Sp.$$

In \cite{hmsref}, Hopkins-Mahowald-Sadofsky proved that the following
statements are equivalent for a $K(n)$-local spectrum $M$.
\begin{enumerate}
    \item $M$ is $K(n)$-locally invertible.
    \item $\dim_{K(n)_\ast}K(n)_\ast M = 1$.
    \item $E^\vee_\ast M$ is a free $E_\ast$-module of rank $1$.
\end{enumerate}
The above discussion suggests that one may recast this result as a
descent-theoretic statement along the \'etale cover
$$\spf E\to \spf E/\Gamma.$$
This is one of the results proved in this paper.

One useful computational tool in the study of the $K(n)$-local Picard group is
the existence of a map
$$\pic_n\to \H^1_c(\Gamma; E_0^\times).$$
This descent-theoretic viewpoint allows us to think of this assignment as the
monodromy action of the line bundle over $\spf E/\Gamma$ corresponding to a
$K(n)$-locally invertible spectrum.

As an approximation to $\pic_n$, one can attempt to understand the Picard group
of the higher real $K$-theories. In the simplest case, one has an
identification
$$\pic(KO) \simeq \Z/8,$$
generated by $\Sigma KO$. This corresponds to the $8$-fold periodicity of $KO$.
Recently, Heard-Mathew-Stojanoska computed in \cite{pichms} that if
$$EO_{p-1} = E^{hC_p}_{p-1},$$
then
$$\pic(EO_{p-1}) \simeq \Z/(2p^2),$$
again generated by $\Sigma EO_{p-1}$. This corresponds to the $2p^2$-fold
periodicity of $EO_{p-1}$. One expects the Picard to be cyclic at any height.
When $p-1$ does not divide $n$ this is a simple computation. In \cite{hhr},
Hill-Hopkins-Ravenel describe the $E_2$-page for the homotopy fixed point
spectral sequence for $EO_{2(p-1)}$. This suggests using tools similar to those
in \cite{pichms} to prove that the Picard group of $EO_{2(p-1)}$ is cyclic. We
will return to this computational problem in a future paper.

Barthel-Beaudry-Stojanoska used this result in \cite{barthelbeaudrystojanoska}
to prove a self-duality statement. Since $\QQ/\Z$ is an injective abelian
group, the functor
$$X\mapsto \Hom(\pi_{-\ast} X, \QQ/\Z)$$
defines a cohomology theory. This is represented by a spectrum $I_{\QQ/\Z}$,
called the Brown-Comenetz dualizing spectrum. The Brown-Comenetz dual of a
spectrum $X$ is defined as
$$I_{\QQ/\Z} X = \underline{\Map}(X, I_{\QQ/\Z}).$$
There is a canonical map $H\QQ\to I_{\QQ/\Z}$, and the fiber of this map is the
Anderson dualizing spectrum, $I_\Z$. One similarly defines the Anderson dual of
a spectrum $X$ to be
$$I_\Z X = \underline{\Map}(X, I_\Z).$$
In \cite{heard-stojanoska}, Heard-Stojanoska showed that there is an
equivalence
$$I_\Z KO \simeq \Sigma^4 KO.$$
Using computational tools, Barthel-Beaudry-Stojanoska proved that, at odd
primes, there is an equivalence
$$\Lk I_{\QQ/\Z} EO_{p-1} \simeq \Sigma^{(p-1)^2} EO_{p-1}.$$
This implies that
$$\Lp I_\Z EO_{p-1} \simeq \Sigma^{(p-1)^2-1} EO_{p-1}.$$
This computational approach does not shed much light (at least to the author)
on the theoretical underpinnings of Anderson self-duality. In this paper, we
provide a conceptual explanation for this fact.

From an algebro-geometric point of view, $I_\Z$ can be thought of as a
dualizing sheaf for $\spec S$. In the first section, we recall some facts about
derived stacks. We then develop methods to analyze dualizing sheaves for even
periodic derived Deligne-Mumford stacks. We prove the following tool for
recognizing when a spectrum is a dualizing sheaf for $\spec S$, which is
tangential to the discussion about Picard groups of $EO_{p-1}$.  Let $R$ be a
coconnected $p$-complete spectrum such that $\pi_\ast R$ is a finite abelian
group for $\ast\neq 1$ and $\pi_0 R$ is a finitely generated abelian group.
Then the following statements are equivalent:
\begin{enumerate}
    \item $\mathrm{Map}(H\Z/p,R) \simeq \Sigma^{-1} H\Z/p$, and
    \item $R$ is a dualizing sheaf for $\spec S$.
\end{enumerate}
In a later paper, we will give an application of this result to a higher Snaith
theorem (see \cite{craig-imj}).

Let us return to Anderson self-duality. Let $G\subseteq \Gamma$ be a finite
subgroup of the Morava stabilizer group. Consider the structure map $f:\spf
E/G\to \spec S$; then $f^! I_\Z$ is exactly $\Lk I_\Z E^{hG}$. Using general
statements about self-duality in the derived setting (see Theorem
\ref{andinvert} and Proposition \ref{lurie-prop}), we deduce that $I_\Z E^{hG}$
is an element of $\pic(E^{hG})$ for any height and prime.

If $G\subseteq \Gamma$ is not a finite group, then our argument does not
necessarily work. However, when $G = \Gamma$, the quasicoherent sheaf $f^!
I_\Z$ on $\spf E/\Gamma$ is in fact $K(n)$-locally invertible, although our
methods do not suffice to give a proof. As there is an equivalence
$$\Sigma^{-1} \widehat{I} \simeq f^! I_\Z,$$
where $\widehat{I}$ is the Gross-Hopkins element of $\pic_n$, this statement is
equivalent to Gross-Hopkins duality (the classical proof is in
\cite{strickland}).

Using the invertibility of this element, we deduce that --- conditional on
$E^{hG}$ being Spanier-Whitehead self-dual, which is proved in Appendix
\ref{appendix} at any height divisible by $(p-1)$ for the subgroup $G = C_p$
--- if the group of exotic elements of $\pic(E^{hG})$, i.e., elements $X$ such
that $E^\vee_\ast E^{hG} \simeq E^\vee_\ast X$ as Morava modules, is cyclic or
trivial, then $\Lk I_\Z E^{hG}$ is equivalent to a shift of $E^{hG}$. Thus, the
result about the cyclicity of the Picard group of $EO_{p-1}$ implies that $\Lp
I_\Z EO_{p-1}$ is equivalent to a shift of $EO_{p-1}$. However, our method does
not give the exact shift of $(p-1)^2-1$.

Gross and Hopkins also describe the monodromy action on the line bundle $f^!
I_\Z$, and show that it is essentially the determinant representation of
$\Gamma$. The question of how one might recover this result using the methods
of this paper is the subject of future work.

\subsection{Acknowledgements}
I'm glad to be able to thank Marc Hoyois for many helpful discussions on the
subject of this paper, and for introducing me to stacks and the
algebro-geometric viewpoint on norm maps in representation theory. I'm also
grateful to Agn\`es Beaudry, Hood Chatham, Jeremy Hahn, Drew Heard, Adeel Khan,
Pax Kivimae, Tyler Lawson, Jacob Lurie, Haynes Miller, Eric Peterson, Paul
VanKoughnett, Craig Westerland, Zhouli Xu, and Allen Yuan for helpful
conversations.

\section{Derived stacks}
Most of the discussion in this section can be found in more detail in
\cite{SAG}.
\subsection{Generalities}
All Deligne-Mumford stacks are assumed to have affine diagonal.
\begin{definition}
    A \emph{derived (Deligne-Mumford) stack} $\dX$ is a Deligne-Mumford stack
    $X$ along with a sheaf of $\Eoo$-rings $\co^{\der}_X$ (interchangeably
    denoted $\co_\dX$) on the affine \'etale site of $X$ such that $\pi_0
    \co^{\der}_X \simeq \co_X$ and $\pi_i \co^{\der}_X$ is a quasicoherent
    $\pi_0 \co^{\der}_X$-module.
\end{definition}
We say that a Deligne-Mumford stack $X$ ``admits a lift'' if there is a derived
stack with underlying stack $X$. This is a rather strong condition to impose on
a Deligne-Mumford stack; see, for instance, \cite{adjoining-roots} and
\cite{roots-of-unity} for results on non-liftability.

Let $\cf$ be a sheaf of $\Eoo$-rings on a Deligne-Mumford stack $Y$, and let
$f:X\to Y$ denote a morphism of Deligne-Mumford stacks. Define a sheaf of
$\Eoo$-rings $f^{-1} \cf$ on $X$ as follows: for every \'etale map $\spec R\to
X$, we define $(f^{-1} \cf)(\spec R)$ to be the homotopy colimit $\colim_{\spec
R\to Z\to Y, Z\to Y} \cf(Z)$ over all such \'etale morphisms.
\begin{definition}
    Let $\dX$ and $\dY$ denote derived stacks. A morphism $\dX\to \dY$ is a
    morphism $f:X\to Y$ along with a morphism $f^{-1} \co_\dX\to \co_\dY$ of
    sheaves of $\EOo$-rings which induces the map $f^{-1} \co_X\to \co_Y$.
\end{definition}
If $f:\dX\to \dY$ is a morphism of derived stacks, then $f^\ast \cf =
\co_\dY\otimes_{f^{-1} \co_\dX} f^{-1} \cf$.

One can define the $\infty$-category of quasicoherent sheaves on derived stacks
just as in the classical case:
$$\QCoh(\dX) = \lim_{\spec R\to \dX} \Mod(\co^{\der}(\spec R)),$$
where the homotopy limit is taken over all \'etale morphisms $\spec R\to \dX$.
This is a limit of presentable stable $\infty$-categories under
colimit-preserving functors, so $\QCoh(\dX)$ is also a presentable stable
$\infty$-category.

Using descent theory, we can give an equivalent presentation. Suppose $\spec
R\to X$ is an \'etale surjection. Then
$$ \QCoh(\dX) \simeq \mathrm{Tot}\left(\xymatrix{ \QCoh(\spec R)
\ar@<0.5ex>[r]\ar@<-0.5ex>[r] & \QCoh(\spec R\times_\dX \spec
R)\ar@<-0.75ex>[r]\ar[r]\ar@<0.75ex>[r] & \cdots}\right).$$
This comes from the presentation of $\dX$ as a semisimplicial object
$$\xymatrix{ \spec R & \ar@<0.5ex>[l]\ar@<-0.5ex>[l] \QCoh(\spec R\times_\dX
\spec R) & \ar@<-0.75ex>[l]\ar[l]\ar@<0.75ex>[l] \cdots}.$$

One way to obtain derived stacks is via the following theorem.
\begin{theorem}\label{lurie-etale}
    Let $R$ be an $\Eoo$-ring. Suppose $f:\pi_0 R\to A$ is an \'etale map of
    ordinary rings; then there is an $R$-algebra $B$ with an \'etale map $R\to
    B$ such that $\pi_0 B \cong A$, and the induced map on homotopy agrees with
    the original map $f$.
\end{theorem}
\begin{proof}
    This theorem can be deduced from work of Goerss and Hopkins in
    \cite{goerss-hopkins}, and can also be found as \cite[Theorem 7.5.0.6]{HA}.
\end{proof}
In what follows, we will be interested in derived formal schemes. To this end,
we make the following definition.
\begin{definition}
    An adic $\Eoo$-ring is an $\Eoo$-ring $R$ with a topology on $\pi_0 R$,
    such that $\pi_0 R$ admits a finitely generated ideal $I$ of definition.
\end{definition}
Let $M$ be a $R$-module. Pick a set of generators $x_1,\cdots,x_n$ for $I$. Say
that $M$ is $(x_i)$-complete if
$$\lim(\cdots\xar{x_i} M\xar{x_i} M\xar{x_i} M) \simeq 0,$$
where $x_i:M\to M$ is the morphism determined by $x_i\in \pi_0 R$. The
$R$-module $M$ is said to be $I$-complete if $M$ is $(x_i)$-complete for $1\leq
i\leq n$.

Let $R, S$, and $T$ be adic $\Eoo$-rings, such that $R$ and $T$ have finitely
generated ideals of definition $I\subseteq \pi_0 R$ and $J\subseteq \pi_0 T$.
Then we can endow $\pi_0(R\wedge_S T)$ with the $K$-adic topology, where $K$ is
the ideal generated by the images of $I$ and $J$. The resulting adic
$\Eoo$-ring is denoted $R\widehat{\otimes}_S T$.

An adic $\Eoo$-ring determined a derived formal scheme $\spf R$, whose
underlying (formal) scheme is $\spf \pi_0 R$. The sheaf $\co^{\der}$ of
$\Eoo$-rings on the affine \'etale site of $\spf \pi_0 R$ is defined as
follows. Let $\spec A\to \spf \pi_0 R$ be an affine \'etale over $\spf \pi_0
R$, given by a map $\pi_0 R\to A$; lift $A$ to an \'etale $R$-algebra $B$ by
Theorem \ref{lurie-etale}. As a functor from $\Aff{X}$ to $\Eoo$-rings, we
define
$$\co^{\der}(\spec A) = B^\wedge_I.$$

More generally, the procedure described above allows us to construct a
quasicoherent sheaf $\cf$ on $\spf R$ from any any $I$-complete $R$-module $M$:
we send
$$\cf(\spec A) = (B\otimes_R M)^\wedge_I.$$
This begets an equivalence
$$\QCoh(\spf R) \simeq \Mod(R)^\wedge_I,$$
where the right hand side denotes the $\infty$-category of $I$-complete
$R$-modules. The smash product of adic $\Eoo$-rings defined above allows us to
consider the fiber product of derived (affine) formal schemes.

For the rest of this paper, any $\Eoo$-ring $R$ will be assumed to be an adic
$\Eoo$-ring with a fixed finitely generated ideal of definition $I$. There
should not be any confusion as to what this ideal is; this will be clear from
the context. Note that every $\Eoo$-ring $R$ can trivially be viewed as an adic
$\Eoo$-ring: endow $\pi_0 R$ with the discrete topology (equivalently, suppose
that $I$ is nilpotent).

Suppose $G$ is a finite group acting on an $\Eoo$-ring $R$ by $\Eoo$-maps. We
can then define the quotient $\spf R/G$ as the colimit of the resulting functor
from $BG$ into the $\infty$-category of formal derived Deligne-Mumford stacks.
Using the cosimplicial model for $BG$ (equivalently, \'etale descent), this can
equivalently be presented via the semisimplicial diagram
$$ \xymatrix{\spf R & \ar@<0.5ex>[l]\ar@<-0.5ex>[l] \spf(R\times G) &
\ar@<-0.75ex>[l]\ar[l]\ar@<0.75ex>[l] \cdots}.$$
We will also need to consider special cases when $G$ is not finite. If
$\dX_\bullet$ is a semisimplicial object in derived stacks, we will denote by
$\QCoh(\dX_\bullet)$ the totalization $\Tot(\QCoh(\dX_\bullet))$ of the
semicosimplicial diagram $\QCoh(\dX_\bullet)$. If $\dX$ is a derived stack, then
$\QCoh(\dX_\bullet^\mathrm{constant}) \simeq \QCoh(\dX)$, where
$\dX_\bullet^\mathrm{constant}$ is the constant semisimplicial object. We
will often abuse notation by using $\dX$ to denote
$\dX_\bullet^\mathrm{constant}$.

\subsection{Vector bundles}
Let $R$ be an $\Eoo$-ring. A projective $R$-module $M$ is a retract of a free
$R$-module. A simple consequence of this definition is that projective
$R$-modules are flat, since direct sums and retracts of flat modules are flat.
In other words, the natural map
$$\pi_0 M \otimes_{\pi_0 R} \pi_\ast R\to \pi_\ast M$$
is an isomorphism.

\begin{definition}\label{vb}
    Let $X$ be a derived stack. A \emph{vector bundle of rank $n$} on $X$ is a
    quasicoherent sheaf $\ff$ such that for every \'etale map $f:\spec R\to X$,
    the pullback $M := f^\ast \ff\in \Mod_R$ satisfies the following
    properties:
    \begin{itemize}
	\item $M$ is a projective $R$-module such that $\pi_0 M$ is a finitely
	    generated $\pi_0 R$-module.
	\item $\pi_0(k\otimes_R M)$ is a $k$-vector space of dimension $n$
	    where $k$ is a field with a map of $\Eoo$-rings $R\to k$.
    \end{itemize}
\end{definition}
A line bundle is a vector bundle of rank $1$. Let $\pic(X)$ be the space of
suspensions of line bundles on $X$, topologized as a subspace of the maximal
subgroupoid inside $\QCoh(X)$. As a corollary of the discussion in \cite[\S
2.9.4-5]{SAG}, we find that if $X$ is a connected derived stack, then $\pic(X)$
is equivalent to the space of invertible objects of the $\infty$-category
$\QCoh(X)$.

Before proceeding, let us discuss how $\pic(\QCoh(\spf R))$ relates to
$\pic(\QCoh(\spec R))$. Suppose $\pi_0 R$ is $I$-complete. It is then clear
that any invertible $R$-module is in $\pic(\QCoh(\spf R))$. Moreover, an
element of $\pic(\QCoh(\spf R))$ is in $\pic(\QCoh(\spec R))$ if and only if
$M$ is a perfect $R$-module.

Let $R$ be an even periodic adic $\Eoo$-ring with ideal of definition $I$ such
that:
\begin{itemize}
    \item $\pi_0 R$ is a complete regular local Noetherian ring which is
	$I$-complete.
    \item An $R$-module is dualizable in $\Mod(R)$ if and only if it is
	perfect.
\end{itemize}
\begin{prop}\label{picinv}
    If $R$ satisfies the above two conditions, then $\pic(\spf R)$ is
    equivalent to the space of invertible objects of $\QCoh(\spf R) \simeq
    \Mod(R)^\wedge_I$.
\end{prop}
\begin{proof}
    By \cite[Theorem 8.7]{baker}, any $R$-module in $\pic(\QCoh(\spec R))$ is
    equivalent to a shift of $R$. Clearly $R$ and $\Sigma R$ are perfect
    $R$-modules, so the second condition on $R$ implies that any invertible
    object of $\QCoh(\spf R)$ is in $\pic(\spf R)$. Conversely, if $M$ is a
    line bundle over $R$, then $M$ is dualizable. Indeed, dualizable objects
    are closed under retracts and wedges, so since $R$ is dualizable, any
    vector bundle over $R$ is dualizable. In particular, $M$ is a perfect
    $R$-module, so it suffices to show that $M$ is an invertible object of
    $\Mod(R)$. Let $M^\vee$ denote the dual of $M$, so there is an evaluation
    map $\ev:M \otimes_R M^\vee \to R$. Now arguing as in \cite[Proposition
    2.9.4.2]{SAG} (which requires \cite[Proposition 2.9.2.3]{SAG}, the proof of
    which does not need $R$ to be connective), we conclude that $\ev$ is an
    isomorphism, so $M$ is an invertible $R$-module.
\end{proof}

It is a general fact that the functor sending a symmetric monoidal
$\infty$-category $\cc$ to the space of invertible objects in $\cc$ commutes
with limits and filtered colimits (\cite[Proposition
2.2.3]{mathew-stojanoska}). By construction, $\QCoh(-)$ sends colimits to
limits of symmetric monoidal stable $\infty$-categories. As the functor from
the $\infty$-category of symmetric monoidal stable $\infty$-categories to the
$\infty$-category of symmetric monoidal $\infty$-categories reflects limits, it
follows that $\pic(-)$ takes homotopy colimits to homotopy limits.

In particular, if $G$ is a finite group acting on $R$ by $\Eoo$-maps, we have
an equivalence (see also \cite[\S 3.3]{mathew-stojanoska}):
$$\pic(\spf R/G) \simeq \pic(\spf R)^{hG}.$$
Note that the $G$-actions on $\pic(\spf R)$ and $\pic(\spec R)$ are the same.

Let $R$ be an even periodic $\Eoo$-ring, and let $M$ be a line bundle over $R$.
Then $\pi_\ast M$ is a projective $\pi_\ast R$-module. Indeed, $\pi_0 M$ is a
projective $\pi_0 R$-module. Since
$$\pi_n M \simeq \pi_n R \otimes_{\pi_0 R} \pi_0 M,$$
the result then follows from $R$ being even periodic and the fact that
projective modules are flat. If, moreover, $\pi_0 R$ is a local ring, then
$\pi_\ast M$ is a free $\pi_\ast R$-module since projective modules over a
local ring are free.

\section{Dualizing sheaves}
\subsection{The connective case}
If $f:X\to Y$ is a morphism of (derived) schemes, we will write $f^!$ to denote
a right adjoint to $f_\ast:\QCoh(X) \to \QCoh(Y)$. This is an abuse of notation
unless $f$ is a proper morphism.
\begin{definition}
    Let $\dX$ be a \emph{connective} derived stack. Let $\cf$ be a quasicoherent
    sheaf over $\dX$. We say that $\cf$ is a \emph{dualizing sheaf} if the
    following conditions are satisfied.
    \begin{enumerate}
	\item The map $\co_\dX\to \underline{\Map}_{\co_\dX}(\omega_\dX,
	    \omega_\dX)$ is an equivalence.
	\item $\omega_\dX$ is coconnected.
	\item $\omega_\dX$ is coherent.
	\item $\omega_\dX$ has finite injective dimension.
    \end{enumerate}
\end{definition}
\cite[Proposition 6.6.2.1]{SAG} shows that if $\cf$ and $\cg$ are two dualizing
sheaves, then there is a line bundle $\cL$ such that $\cf\simeq \cg\otimes\cL$.
We give a simple tool to identify dualizing sheaves over the ($p$-complete)
sphere spectrum.
\begin{theorem}\label{andersonid}
    Let $R$ be a coconnected $p$-complete spectrum such that $\pi_\ast R$ is a
    finite abelian group for $\ast\neq 1$ and $\pi_0 R$ is a finitely generated
    abelian group. Then the following statements are equivalent:
    \begin{enumerate}
	\item $\underline{\Map}(H\Z/p,R) \simeq \Sigma^{-1} H\Z/p$, and
	\item $R$ is a dualizing sheaf for $\spec S$.
    \end{enumerate}
\end{theorem}
\begin{proof}
    It is easy to see that the Anderson dualizing spectrum $I_\Z$ (described in
    the introduction) is a dualizing $S$-module.
    
    Returning to the theorem, assume (2). The above discussion implies that any
    dualizing sheaf is equivalent to $I_\Z$ up to an element of $\pic(\Sp)$,
    which is isomorphic to $\Z$ generated by $S^1$ (see Lemma \ref{picard}).
    Without loss of generality, we may assume that $R = I_\Z$; then $R$ sits in
    a fiber sequence $R\to I_\QQ\to I_{\QQ/\Z}$, which implies that
    $$\underline{\Map}(H\Z/p,R) \simeq \Sigma^{-1}
    \underline{\Map}(H\Z/p,I_{\QQ/\Z}) \simeq \Sigma^{-1} H\Z/p.$$
    
    For the other direction, assume $R$ satisfies (1). Let $K$ be a dualizing
    sheaf for $\spec S$; translating the definition provided above, this means
    that $K$ is a spectrum such that
    \begin{itemize}
	\item[(a)] $K$ is coconnected, and $\pi_n K$ is a finitely generated
	    abelian group.
	\item[(b)] $K$ has finite injective dimension, i.e., there is an
	    integer $n$ such that for any $n$-coconnected spectrum $M$, we have
	    $\pi_i \underline{\Map}(M,K) = 0$ for $i<0$.
	\item[(c)] The natural map $S\to \underline{\Map}(K,K)$ is an
	    equivalence.
    \end{itemize}
    To show that $R$ is a dualizing sheaf for $\spec S$, we will check each of
    the conditions above.
    \begin{itemize} 
	\item[(a)] $R$ is coconnected by assumption, and $\pi_n R$ is a
	    finitely generated abelian group for all $n$.
	\item[(b)] We have $\pi_i \underline{\Map}(M, R) \simeq \pi_0
	    \underline{\Map}(\Sigma^i M, R)$. Suppose $R$ is $N$-coconnected
	    for some $N$; then
	    $$\pi_i \underline{\Map}(M, R) \simeq \pi_0
	    \underline{\Map}(\tau_{\leq N}\Sigma^i M, R).$$
	    Now, $\pi_k \Sigma^i M\simeq 0$ for $k > n + i$. If $n$ is
	    sufficiently large, then $\tau_{\leq N}\Sigma^i M$ is contractible,
	    so $\underline{\Map}(\pi_i M, R) \simeq 0$ for some $n\gg 0$.
    \item[(c)] Our proof follows \cite[Proposition 6.6.4.6]{SAG}. It suffices
	to prove that for every integer $k$, we have an equivalence $\pi_k S\to
	    \pi_k \underline{\Map}(R, R) \simeq \pi_k
	    \underline{\Map}(\underline{\Map}(S, R), R)$. Since $R$ is
	    $N$-coconnected, we can replace $S$ by its $N$-coconnected cover.
	    In this case, $\tau_{\leq N} S$ can be written as a composite of
	    extensions of $H\Z_p$ and shifts of Eilenberg-Maclane spectra
	    annihilated by a power of $p$, i.e., $H\Z/p^k$.
        
	    It therefore suffices to show that
	    $\underline{\Map}(\underline{\Map}(H\Z/p^k,R),R) \simeq H\Z/p^k$
	    and that $\underline{\Map}(\underline{\Map}(H\Z_p,R),R) \simeq
	    H\Z_p$. But
	    \begin{align*}
		\underline{\Map}(\underline{\Map}(H\Z/p^k,R),R) & \simeq
		\underline{\Map}(\underline{\Map}_{H\Z/p}(H\Z/p^k,R'),R') \\
		& \simeq
		\underline{\Map}_{H\Z/p}(\underline{\Map}_{H\Z/p}(H\Z/p^k,R'),R'),
	    \end{align*}
	    which is $H\Z/p^k$ since $R'$ is a dualizing sheaf for $\spec
	    H\Z/p$, where $R' = \underline{\Map}(H\Z/p,R)$; in particular, this
	    proves that $M\to \underline{\Map}(\underline{\Map}(M,R),R)$ is an
	    equivalence for every spectrum $M$ which is $p$-torsion. For
	    $H\Z_p$, we argue as follows: there is an equivalence
	    $\underline{\Map}(H\Z_p,R) \simeq \underline{\Map}(\Sigma^{-1}
	    H\QQ_p/\Z_p,R)$ since $R$ is torsion. We are now done by the
	    previous case.
    \end{itemize}
\end{proof}

The theory of dualizing sheaves over connective derived Deligne-Mumford stacks
is not sufficient for our purposes; we have to extend the definition to
even periodic derived stacks. Recall the following definition.
\begin{definition}
    An even periodic $\Eoo$-ring is an $\Eoo$-ring $R$ whose homotopy is
    concentrated in even dimensions such that $\pi_2 R$ is an invertible $\pi_0
    R$-module, satisfying the property that $\pi_{2k} R \simeq (\pi_2
    R)^{\otimes k}$ for all $k\in \Z$.
\end{definition}
\begin{definition}
    An even periodic derived stack $\dX$ is a derived stack such that for every
    \'etale morphism $\spec R\to X$ into the underlying Deligne-Mumford
    stack, the $\Eoo$-ring $\co_\dX(\spec R)$ is a even periodic $\Eoo$-ring.
\end{definition}
\begin{remark}
    Let $X$ be a Deligne-Mumford stack with a flat map $X\to \Mfg$. An even
    periodic refinement of $X$ is an even periodic derived stack $\dX$ lifting
    $X$ such that for every \'etale morphism $\spec R\to X$, the even periodic
    $\EOo$-ring $\co_\dX(\spec R)$ has formal group given by the (flat)
    composite $\spec R \to X\to \Mfg$.
\end{remark}

\subsection{The even periodic case}

If $R$ is an $\Eoo$-ring, the notion of an almost perfect $R$-module is only
well-defined when $R$ is connective. In the nonconnective setting, we will make
the following definition.
\begin{definition}
    Let $R$ be a Noetherian even periodic $\Eoo$-ring. An $R$-module $M$ is
    said to be almost perfect if it can be obtained as the geometric
    realization of a simplicial $R$-module $P_\bullet$, with each $P_n$ a free
    $R$-module of finite rank.
\end{definition}
If $\dX$ is a locally Noetherian even periodic derived stack, then a
quasicoherent sheaf $\cf$ on $\dX$ will be called almost perfect if, for every
\'etale morphism $f:\spec R\to \dX$, the pullback $f^\ast \cf$ is almost
perfect.

The definition of a dualizing sheaf is the following.
\begin{definition}\label{dualizing-sheaves}
    Let $\dX$ be a locally Noetherian even periodic derived stack. A
    quasicoherent sheaf $\omega_\dX$ on $\dX$ is a \emph{dualizing sheaf} if
    \begin{enumerate}
	\item The map $\co_\dX\to \underline{\Map}_{\co_\dX}(\omega_\dX,
	    \omega_\dX)$ is an equivalence.
	\item The functor $\DD(\cf) =
	    \underline{\Map}_{\co_\dX}(\cf,\omega_\dX)$ gives an
	    autoequivalence of the category of almost perfect quasicoherent
	    sheaves on $\dX$ with itself.
	\item For every \'etale map $f:\spec R\to \dX$, the $\pi_0 R$-module
	    $\pi_0 f^\ast \omega_\dX$ is a dualizing module for $\pi_0 R$.
    \end{enumerate}
\end{definition}

We will need to understand when the structure sheaf (or some shift of it) of a
derived stack $\dX$ is itself a dualizing complex.  If this is the case, we say
that $\dX$ is \emph{self-dual} or \emph{Gorenstein}.

We begin with a series of lemmas.
\begin{lemma}\label{etaledual}
    Let $f:\dX\to\dY$ be a \'etale surjection of locally Noetherian even
    periodic derived Deligne-Mumford stacks. Suppose that $\omega_\dY$ is a
    quasicoherent sheaf on $\dY$ such that $f^\ast \omega_\dY$ is a dualizing
    sheaf on $\dX$. Then $\omega_\dY$ is a dualizing sheaf on $\dY$.
\end{lemma}
\begin{proof}
    Condition (1) is obvious, and condition (2) follows from the fact that
    $f^\ast$ preserves almost perfectness. It remains to check condition (3).
    Let $g:\spec R \to \dY$ be an \'etale map. We need to check that $\pi_0
    g^\ast \omega_\dY$ is a dualizing module for $\pi_0 R$. The statement of
    Lemma \ref{etaledual} is true in the classical setting, so it suffices to
    check that $p_0^\ast \pi_0 g^\ast \omega_\dY$ is a dualizing sheaf on $T$,
    for some \'etale surjection $p_0: T\to \spec \pi_0 R$. Let $\dZ$ denote the
    even periodic Deligne-Mumford stack $\dX\times_\dY \spec R$, and let $\spec
    A\to \dZ$ be an \'etale surjection. Let $q:\spec A\to \dX$ denote the
    induced \'etale morphism. The map $\dZ\to \spec R$ is also an \'etale
    surjection, so the composite $p:\spec A\to \spec R$ is an \'etale
    surjection. Since $p$ is \'etale, we have an isomorphism $p_0^\ast \pi_0
    g^\ast \omega_\dY \simeq \pi_0 p^\ast g^\ast \omega_\dY$. The equivalence
    $p^\ast g^\ast = q^\ast f^\ast$ shows that $p_0^\ast \pi_0 g^\ast
    \omega_\dY$ is equivalent to $\pi_0 q^\ast f^\ast \omega_\dY$ as a $\pi_0
    A$-module. Since $f^\ast \omega_\dY$ is a dualizing sheaf for $\dX$, and
    $q$ is an \'etale morphism, it follows that $\pi_0 q^\ast f^\ast
    \omega_\dY$ is a dualizing module for $\pi_0 A$, as desired.
\end{proof}

\begin{lemma}\label{IZ}
    Suppose $\dX$ is a locally Noetherian separated derived Deligne-Mumford
    stack which arises as an even-periodic refinement of a tame and flat map
    $X\to \Mfg$. Assume that $\dX$ is perfect and proper. Let $f:\dX\to \spec
    S$ be the structure morphism. Then $f^! I_\Z$ is a dualizing sheaf on
    $\dX$.
\end{lemma}
\begin{proof} 
    We will check the conditions of Definition \ref{dualizing-sheaves}.
    \begin{enumerate}
	\item We need to show that the map $\co_\dX\to
	    \underline{\Map}_{\co_\dX}(f^! I_\Z, f^! I_\Z)$ is an equivalence,
	    i.e., that for each $\cf\in \QCoh(\dX)$, the map
	    $$\theta:\Map_{\dX}(\cf, \co_\dX) \to \Map(f_\ast(\cf\otimes f^!
	    I_\Z), I_\Z)$$
	    is an equivalence. The same proof as \cite[Proposition
	    6.6.3.1]{SAG} can be used here, but we will recall the details for
	    the sake of completeness. As $\dX$ is a perfect stack, the sheaf
	    $\cf$ is a filtered colimit of perfect objects. In particular, we
	    may assume that $\cf$ is perfect, so that
	    $$f_\ast(\cf\otimes f^! I_\Z) \simeq \underline{\Map}(f_\ast
	    \cf^\vee, I_\Z).$$
	    Since $I_\Z$ is a dualizing complex for $\spec S$, it suffices to
	    show that $f_\ast \cf^\vee$ is almost perfect over the sphere. It
	    suffices to know that the global sections functor (to
	    $\Gamma(\dX,\co_\dX)$-modules) preserves filtered colimits, which
	    follows from \cite[Theorem 4.14]{mathew-meier}.
	\item We need to show that there is an equivalence $\cf\xar{\simeq}
	    \underline{\Map}_{\co_\dX}(\underline{\Map}_{\co_\dX}(\cf, f^!
	    I_\Z), f^! I_\Z)$ for every almost perfect quasicoherent sheaf
	    $\cf$. The assertion is local, so we may assume that $\dX = \spec
	    R$ is affine. In this case, the result follows from the fact that
	    for any almost perfect $R$-module $M$, the map $M\to
	    \underline{\Map}_R(\underline{\Map}_R(M, I_\Z R), I_\Z R)$ is an
	    equivalence.
	\item Let $u:\spec R\to \dX$ be an \'etale morphism. We need to show
	    that $\pi_0 u^\ast f^! I_\Z$ is a dualizing sheaf for $\pi_0 R$.
	    The $R$-module $u^\ast f^! I_\Z$ is the function spectrum
	    $\underline{\Map}(R, I_\Z) = I_\Z R$. There is a short exact
	    sequence
	    $$0\to \Ext^1_\Z(\pi_{-n-1} R, \Z) \to \pi_n I_\Z R \to
	    \Hom_\Z(\pi_{-n} R, \Z)\to 0.$$
	    Since $R$ is an even periodic cohomology theory, it follows
	    that $\pi_n I_\Z R \simeq \Hom_\Z(\pi_{-n} R, \Z)$, so $\pi_0 I_\Z
	    R$ is indeed a dualizing sheaf for $\pi_0 R$, as desired.
    \end{enumerate}
\end{proof}

\begin{remark}
    Suppose $\dY$ is a Deligne-Mumford stack for which the structure morphism
    $\dY\to \spec S$ factors as $\dY\hookrightarrow \dX\to \spec S$, where
    $u:\dY\hookrightarrow \dX$ is an open immersion, and $f:\dX\to \spec S$ is
    a Deligne-Mumford stack satisfying the conditions of Lemma \ref{IZ}. Then
    $u^\ast f^! I_\Z$ is a dualizing sheaf on $\dY$. In the classical setting,
    Nagata's compactification theorem gives such a factoring when $\dY$ is a
    scheme which is separated and of finite type. We do not know of an analogue
    of this result in the derived setting.
\end{remark}

\begin{remark}\label{redemption}
    Lemma \ref{IZ} is also true if $\dX$ is replaced by $\spf E$, where $E$ is
    a Morava $E$-theory (see Section \ref{e-thy}).
\end{remark}

We say that a locally Noetherian Deligne-Mumford stack $X$ has finite global
dimension if there is a finite \'etale cover $\spec R\to X$ with $R$ a
Noetherian ring of finite global dimension.
\begin{lemma}\label{uniqueness-up-to-line-bundles}
    Let $\dX$ be a locally Noetherian even periodic derived Deligne-Mumford
    stack of finite global dimension\footnote{In general, this is stronger than
    having finite Krull dimension.}, and let $\omega$ be a dualizing sheaf on
    $\dX$. A quasicoherent sheaf $\omega'$ on $\dX$ is a dualizing sheaf if and
    only if there is an equivalence $\omega' \simeq \omega\otimes \cL$ for
    $\cL$ a line bundle on $\dX$.
\end{lemma}
\begin{proof}
    Suppose $\cL$ is a line bundle on $\dX$, and let $\omega' = \omega \otimes
    \cL$. Conditions (1) and (2) of Definition \ref{dualizing-sheaves} are
    immediate. Suppose $f:\spec R\to \dX$ is an \'etale morphism. There is an
    isomorphism $\pi_0 f^\ast \omega' \simeq \pi_0 f^\ast \omega \otimes_{\pi_0
    R} \pi_0 \cL$. Since $\cL$ is a line bundle, $\pi_0 \cL$ is a line bundle
    over $\pi_0 R$, so $\pi_0 f^\ast \omega'$ is a dualizing sheaf on $\pi_0
    R$, establishing condition (3).

    The proof of the converse follows \cite[Proposition 6.6.2.1]{SAG}. Suppose
    $\omega$ and $\omega'$ are dualizing sheaves. Let $\cL =
    \underline{\Map}_{\co_\dX}(\omega, \omega')$. We will show that $\cL$ is a
    line bundle, and that $\omega\otimes \cL \simeq \omega'$. Suppose $\cf$ is
    an almost perfect quasicoherent sheaf on $\dX$. We will first show that
    $$\cf\otimes \cL\to
    \underline{\Map}_{\co_\dX}(\underline{\Map}_{\co_\dX}(\cf, \omega),
    \omega')$$
    is an equivalence. To show this, it in turn suffices to show that, if $R$
    is an even periodic Noetherian $\EOo$-ring of finite global dimension, and
    $K$ and $K'$ are dualizing complexes, then for every almost perfect
    $R$-module $M$, the map
    $$M\otimes_R \underline{\Map}_R(K,K') \to
    \underline{\Map}_R(\underline{\Map}_R(M,K),K')$$
    is an equivalence. Since the statement is true for perfect $R$-modules, it
    suffices to reduce the result to this case.

    If $M$ is almost perfect, then $\pi_0 M$ and $\pi_1 M$ are both finitely
    generated $\pi_0 R$-modules. The statement for $\pi_0 M$ is clear from the
    definition. Choose a finitely generated free module $P\to M$ inducing the
    surjection $\pi_0 P\to \pi_0 M$. The fiber $P'$ of $P\to M$ is almost
    perfect, so $\pi_0 P'$ is also finitely generated. The long exact sequence
    in homotopy gives a short exact sequence
    $$0\to \coker(\pi_1 P'\to \pi_1 P) \to \pi_1 M\to \ker(\pi_0 P'\to \pi_0
    P)\to 0.$$
    The $\pi_0 R$-modules $\coker(\pi_1 P'\to \pi_1 P)$ and $\ker(\pi_0 P'\to
    \pi_0 P)$ are finitely generated, so $\pi_1 M$ is also finitely generated.
    By \cite[Proposition 2.1]{mathew-thick-subcat}, the $R$-module $M$ is
    perfect.

    Having established that $\cf\otimes \cL\to
    \underline{\Map}_{\co_\dX}(\underline{\Map}_{\co_\dX}(\cf, \omega),
    \omega')$ is an equivalence for any almost perfect sheaf $\cf$, it follows
    that, setting $\cf = \underline{\Map}_{\co_\dX}(\omega',\omega)$, there is
    an equivalence $\cf\otimes \cL \simeq \co_{\dX}$, so $\cL$ is a line
    bundle. Moreover, the same result, when applied to $\cf = \omega$, shows
    that $\omega\otimes \cL \simeq \omega'$, as desired.
\end{proof}

\begin{remark}\label{underlying-gorenstein}
    Suppose $\dX$ satisfies the conditions of Lemma
    \ref{uniqueness-up-to-line-bundles}. Let $X$ denote the underlying stack of
    $\dX$. Then $\dX$ is self-dual if and only if $X$ is Gorenstein. Indeed,
    suppose $\dX$ is self-dual. Let $f_0:\spec R \to X$ be an \'etale map. This
    refines to an \'etale map $f:\spec \wt{R}\to \dX$. By construction, $\pi_0
    \wt{R} = R$. It follows from the definition that the $R$-module $\pi_0
    f^\ast \co_\dX = \pi_0 \wt{R} = R$ is a dualizing module, as desired. For
    the converse, suppose $X$ is Gorenstein, and let $\omega_\dX = \co_\dX$.
    Condition (1) of Definition \ref{dualizing-sheaves} is immediate.  To prove
    condition (2), note that the proof of Lemma
    \ref{uniqueness-up-to-line-bundles} shows that an almost perfect
    quasicoherent sheaf $\cf$ on $\dX$ is perfect, in which case the condition
    is easy to establish. Finally, condition (3) follows from the assumption
    that $X$ is self-dual.
\end{remark}

The above discussion yields the following result.
\begin{theorem}\label{andinvert}
    Let $\dX$ be a perfect and proper locally Noetherian separated derived
    Deligne-Mumford stack which arises as the even-periodic refinement of a
    tame and flat map $X\to \Mfg$, where $X$ has finite global dimension. If
    $\dX$ is self-dual, then $f^! I_\Z$ is invertible, where $f:\dX\to\spec S$
    is the structure map.
\end{theorem}

We will also prove the following result, which we learnt from Jacob Lurie.
\begin{prop}\label{lurie-prop}
    Let $f_0: X\to \spec \Z_p$ be a smooth and proper scheme of relative
    dimension $d$. Suppose $f:\dX\to \spec S$ is an even-periodic refinement of
    $X$. Then $f^! I_\Z$ is in $\pic(\dX)$.
\end{prop}
\begin{proof}
    Denote by $g:\tau_{\geq 0} \dX \to \spec S$ the connective cover of $\dX$;
    there are morphisms $i: X\to \tau_{\geq 0} \dX$ and $j: \dX\to \tau_{\geq
    0} \dX$. By Serre duality, if $f_0: X\to \spec \Z_p$ is of relative
    dimension $d$, then $f_0^! \Z_p$ is isomorphic to the line bundle
    $\omega_X$ shifted up to degree $d$. Therefore $f_0^! \Z_p\in \pic(X)$.
    There is a commutative diagram
    \begin{align*}
	\xymatrix{
	    X\ar[r]^i \ar[d]^{f_0} & \tau_{\geq 0} \dX\ar[d]^g\\
	    \spec \Z_p \ar[r]_q & \spec S.
	    }
    \end{align*}
    It follows that
    $$f_0^! \Z = f_0^! q^! I_\Z = i^! g^! I_\Z,$$
    so $i^! g^!  I_\Z \in \pic(X)$. It is easy to see that this implies that
    the sheaf $g^! I_\Z$ on $\tau_{\geq 0} \dX$ looks like $\omega_X[d] \oplus
    \omega_X[d-2] \oplus \cdots$, so that $f^! I_\Z = \Hom_{\co_{\tau_{\geq 0}
    \dX}}(\co_\dX, g^! I_\Z)$ looks like a $2$-periodic version of $\omega_X$,
    concentrated in degrees of the same parity as $d$. In particular, $f^!
    I_\Z$ is in $\pic(\dX)$, as desired.
\end{proof}

\section{$E$-theory}\label{e-thy}
Let $\GG$ be a formal group of height $n$ over a perfect field $k$ of
characteristic $p>0$.
\begin{definition}
    An $\Eoo$-ring $E$ is said to be a \emph{Morava $E$-theory} if the
    following conditions are satisfied:
    \begin{enumerate}
	\item $E$ is even periodic, with a(n invertible) periodicity generator
	    $\beta\in \pi_2 E$.
	\item $\pi_0 E$ is a complete local Noetherian ring with residue field
	    $k$.
	\item The formal group $\spf \pi_0 \underline{\Map}_S(\Sigma^\infty_+
	    \CP^\infty, E)$ over $\pi_0 E$ is a universal deformation of $\GG$.
    \end{enumerate}
\end{definition}
\emph{A priori}, it is not clear that Morava $E$-theory exists; however, it is
a theorem of Goerss-Hopkins-Miller that every pair $(k, \GG)$ of an
perfect field $k$ along with a finite height formal group begets a Morava
$E$-theory $E$. The choice of $k$ and $\GG$ will be remain implicit.

A theorem of Lazard's says that all formal groups of the same height are
isomorphic over an algebraically closed field $k$ of characteristic $p$. A
particular choice for a formal group law of height $n$ is the Honda formal
group law $\H_n$ over $k$, whose $p$-series is given by
$$[p]_{\Gamma_n}(x) = x^{p^n}.$$
By Dieudonn\'e theory, one can show that the profinite group $\SS_n$ of
automorphisms of $\H_n$ over $k$ is given by the units in the maximal order
$\co_n$ of the central division algebra of Hasse invariant $1/n$ over $\QQ_p$.
Explicitly,
$$\SS_n \cong \left(W(k)\langle S\rangle/(Sx = \phi(x) S, S^n =
p)\right)^\times,$$
where $\phi$ is a lift of Frobenius to $W(k)$ and $x\in W(k)$. As $\H_n$ is
defined over $\FF_{p^n}$, we can construct the semidirect product $\SS_n\rtimes
\Gal(k/\FF_{p^n})$; we will call this the Morava stabilizer group, and denote
it by $\Gamma$.

For $N\geq 1$, we have normal subgroups $1+S^N\co_n$ of $\Gamma$, which are of
finite index. Moreover, we have
$$\bigcap_{N\geq 1} (1+S^N \co_n) = 1,$$
so letting these be a basis for the open neighborhoods of $1$ provides $\Gamma$
the structure of a profinite group.

Goerss-Hopkins-Miller showed that the action of $\Gamma$ on $\pi_0 E$ lifts to
an action of $\Gamma$ on the $\Eoo$-ring by $\Eoo$-maps. Choosing $\GG = \H_n$,
Lubin-Tate theory allows us to noncanonically identify
$$\pi_0 E \simeq W(k)[[u_1,\cdots,u_{n-1}]].$$
This is a complete local ring, with maximal ideal $\fr{m} = (p, u_1, \cdots,
u_{n-1})$. We remark that there are explicit, but inhumanly complicated,
formulas for the action of $\Gamma$ on the generators $u_i$.

The $\Eoo$-ring $E$ is therefore an adic $\Eoo$-ring, complete with respect to
the finitely generated ideal $(p, u_1, \cdots, u_{n-1})$. The action of the
Morava stabilizer group on $E$ is continuous in the sense that it acts via maps
of \emph{adic} $\Eoo$-rings.

\begin{theorem}[Devinatz-Hopkins]
    The \emph{continuous} homotopy fixed points $E^{h\Gamma}$ is equivalent to
    the $K(n)$-local sphere $\Lk S$.
\end{theorem}
Working through the definition of the homotopy fixed points, this is saying
that
$$ \Lk S \simeq \Tot\left(\xymatrix{E \ar@<0.5ex>[r]\ar@<-0.5ex>[r] &
E\widehat{\wedge} E \ar@<-0.75ex>[r]\ar[r]\ar@<0.75ex>[r] & \cdots}\right) $$

As $\Gamma$ acts continuously on $E$, we can form the quotient stack $\spf E/G$
for any finite subgroup $G\subsetneq \Gamma$. However, we cannot immediately
define the quotient stack $\spf E/\Gamma$ in the same manner as above; instead,
inspired by the above result of Devinatz-Hopkins, we make the following
definition.
\begin{definition}
    The derived Lubin-Tate stack $\dX$ is defined to be the semisimplicial
    stack $\spf E/\Gamma$, described via the semisimplicial diagram
    $$\xymatrix{\spf E & \ar@<0.5ex>[l]\ar@<-0.5ex>[l] \spf(E\widehat{\wedge}
    E) & \ar@<-0.75ex>[l]\ar[l]\ar@<0.75ex>[l] \cdots}$$
\end{definition}

The following result is the analogue of the identification
$$\QCoh(\spf E/G) \simeq \Mod(E)^{\wedge, G}_\frakm \simeq (\Lk
\Mod(E))^{hG},$$
where $\frakm = (p, u_1, \cdots, u_{n-1})$ and $G$ is a finite subgroup of
$\Gamma$.

\begin{lemma}\label{qcoh}
There are symmetric monoidal equivalences
$$\QCoh(\spf E) \simeq \Lk \Mod(E),\quad \QCoh(\dX) \simeq \Lk \Sp.$$
\end{lemma}
\begin{proof}
    To prove the first statement, it suffices to prove that an $E$-module is
    $\frakm$-complete if and only if it is $K(n)$-local. This follows from
    \cite[Chapter 6, Proposition 4.1]{tmf}. As $v_n = u^{(p^n-1)}$, we can
    invert $u$ in a $K(n)$-local $E$-module; the statement that $K(n)$-local is
    equivalent to $\fr{m}$-complete then follows from \cite[Corollary
    7.3.3.3]{SAG}.
    
    The second equivalence is a formal consequence of descent. Indeed, we have
    an equivalence:
    $$ \Lk \Sp\simeq \mathrm{Tot}\left(\xymatrix{ \QCoh(\spf E)
    \ar@<0.5ex>[r]\ar@<-0.5ex>[r] & \QCoh(\spf E\widehat{\wedge}
    E)\ar@<-0.75ex>[r]\ar[r]\ar@<0.75ex>[r] & \cdots}\right) $$
    Since $\spf E \to \dX$ is a $\Gamma$-Galois \'etale cover and
    $\spf(E\widehat{\wedge} E) \simeq \spf E\times_{\dX} \spf E$, it follows
    that the cosimplicial diagram is the cobar construction for homotopy fixed
    points. Altogether, this means that
    $$\QCoh(\dX) \simeq \QCoh(\spf E)^{h\Gamma},$$
    giving the desired equivalence $\QCoh(\dX) \simeq \Lk \Sp$.
\end{proof}
Note that the map $f^\ast:\QCoh(\spec S) \simeq \Sp \to \QCoh(\dX)$ induced by
the structure map $f:\dX\to \spec S$ is exactly $K(n)$-localization. It is
important to remark here that the na\"ive guess that $\dX$ is $\spf L_{K(n)} S$
is not correct. For instance, let $L_{K(1)} S$ denote the $K(1)$-local sphere,
with the $p$-adic topology. By \cite[Corollary 8.2.4.15]{SAG}, we know that
$$\QCoh(\spf L_{K(1)} S)\simeq \Mod(L_{K(1)} S)^\wedge_p;$$
but this is not equivalent to $L_{K(1)} \Sp\simeq \QCoh(\dX)$.

Vector bundles on $\spf E/G$ when $G$ is a finite subgroup of
$\Gamma$ are ``easy''. Suppose $X = \spf E$; then every vector bundle is a
perfect $E$-module. Our goal in this section is to study vector bundles over
the quotient stack $\spf E/G$ for $G\subseteq \Gamma$ a finite subgroup. This
is equivalent to studying the $\infty$-category of perfect $E$-modules with a
$G$-action.
\begin{prop}\label{thicksubcat}
    The $\infty$-category of vector bundles on $\spf E/G$ is generated by $E[G]
    = E\wedge \Sigma^\infty_+ G$ as a thick subcategory.
\end{prop}
\begin{proof}
    Let $M$ be a perfect $E$-module with a $G$-action. Since $\pi_0 E$ is a
    local ring, $\pi_\ast M$ is a (finitely generated) free $\pi_\ast
    E[G]$-module. Let $x_1,\cdots,x_m$ be a basis for $\pi_\ast M$ over
    $\pi_\ast E[G]$; this begets a map
    $$f:E[G]^{\vee k}\vee \Sigma E[G]^{\vee n} \to M,$$
    which is a surjection on homotopy. The fiber of $f$ is also a free
    $E[G]$-module $E[G]^{\vee i}\vee \Sigma E[G]^{\vee j}$. Therefore, if $K$
    is the cofiber of
    $$E[G]^{\vee i}\to \Sigma E[G]^{\vee k}$$
    and $L$ is the cofiber of
    $$\Sigma E[G]^{\vee j}\to \Sigma E[G]^{\vee n},$$
    we have a splitting of $M$ as $K \vee \Sigma
    L$.

    We provide an alternative proof in the case that $p\not|\#G$. Let $M$ be
    any perfect $E$-module with a $G$-action. We claim that $M$ is a retract of
    $M\widehat{\wedge}_E E[G]$.  Indeed, we have maps $\pi_1:E[G]\to E$ (coming
    from $G\to \ast$) and $\pi_2:E\to E[G]$ (coming from the basepoint).
    Moreover, our assumptions imply that $\# G$ is invertible in $(\pi_0
    E)^\times\supseteq \Z_p^\times$, so $\frac{1}{\#G}\pi_2\pi_1$ gives an
    idempotent map from $E[G]$ to itself.  The image is $E$, which establishes
    that $E$ is a retract of $E[G]$, and hence the claim. To finish the proof
    of the proposition, we note that $M\widehat{\wedge}_E E[G]$ is in the thick
    subcategory generated by $E[G]$; since $M$ is a retract, the desired result
    follows.
\end{proof}
One can ask for more satisfying descriptions along the lines of the following
result of Bousfield's.
\begin{theorem}[Bousfield]
    Every vector bundle over $\spf K_2/C_2$ is a direct sum of suspensions of
    $KO_2$, $K_2$, and $KT = K_2^{h\Z}$.
\end{theorem}
\begin{remark}
    There are two avenues for generalization.
    \begin{enumerate}
	\item One can attempt to describe all vector bundles over $\spf
	    E_{p-1}/C_p$. At odd primes, there are a lot more indecomposable
	    representations. Nonetheless, a partial generalization of
	    Bousfield's result is the subject of ongoing work by Hood Chatham.
	\item One can attempt to prove Bousfield's result in the equivariant
	    setting. In \cite{wood}, we describe a genuine $G$-equivariant
	    generalization of this result for finite abelian groups $G$.
    \end{enumerate}
\end{remark}
\section{Picard groups and Anderson duality}
We now turn our attention to understanding Picard groups.
\subsection{The $K(n)$-local Picard group}
\begin{lemma}\label{picard}
    There is an isomorphism
    $$\pi_0\pic \Sp\simeq \Z\simeq \langle S^1\rangle.$$
\end{lemma}
\begin{proof}
    Let $X\in\pic\Sp$.  Then $X$ is a finite spectrum (i.e., is compact), since
    the sphere is.  We might assume that $X$ is connective with $\pi_0 X\neq
    0$.  The K\"unneth formula tells us that $Hk_\ast X$ is concentrated in
    degree $0$ for every field $k$.  It follows from the universal coefficients
    theorem that $H\Z_\ast X$ is torsion-free and concentrated in degree zero.
    Using the Hurewicz theorem, we can conclude that $X\simeq S$.
\end{proof}
We could now attempt to understand the Picard space of $\Lk \Sp$ -- or, perhaps
a simpler task, the Picard group of $\Lk \Sp$.  This category is not symmetric
monoidal under the ordinary smash product; rather, one has to consider a
completed smash product. For this, we have the following calculation due to
Hopkins-Mahowald-Sadofsky (\cite{hmsref}).
\begin{theorem}\label{piccomputation}
    At an odd prime\footnote{An analogous result is true at $p=2$; there, we
    have $$\pi_0 \pic(\Lone \Sp)\simeq \Z_2\times\Z/2\times\Z/2,$$ generated by
    the elements described below and the ``dual question mark complex''.  },
    there is an isomorphism
    $$\pi_0 \pic(\Lone \Sp)\simeq \Z_p\times\Z/|v_1|.$$
\end{theorem}

As this is really the only computation that is known in general, we will sketch
the proof. This relies on the following incredible theorem, again by
Hopkins-Mahowald-Sadofsky, a geometric proof of which is the goal of this
section.
\begin{theorem}[Hopkins-Mahowald-Sadofsky \cite{hmsref}]\label{hms}
    The following conditions are equivalent.
    \begin{enumerate}
	\item A $K(n)$-local spectrum $M$ is in $\pic\Lk\Sp$.
	\item $\dim_{K(n)_\ast}K(n)_\ast M = 1$.
	\item $E^\vee_\ast M$ is a free $E_\ast$-module of rank $1$.
    \end{enumerate}
\end{theorem}
It is worthwhile to remark that since $E_\ast$ is a complete local ring, the
last condition is equivalent to $E^\vee_\ast M$ being an invertible
$E_\ast$-module.

Let $M(p^k)$ denote the spectrum obtained by taking the cofiber of
$S^{-1}\xar{p^k}S^{-1}$.  There are maps $M(p^k)\to M(p^{k+1})$, which, in the
limit, give a spectrum $M(p^\infty)$.  This is an invertible spectrum: it sits
in a cofiber sequence
$$S^{-1}\to p^{-1} S^{-1}\to M(p^\infty),$$
and multiplication by $p$ annihilates $K(n)$-homology for $n>0$, so that $\Lk
M(p^\infty)\simeq \Lk S$ -- this certainly has $K(n)$-homology of dimension
$1$. Since $M(p^k)$ is a finite spectrum, it is of type $k$ for some integer
$k$.  A theorem of Adams says that $k=1$.  By the periodicity theorem, we
therefore obtain a $v_1$-self map
$$v_1^{p^{k-1}}:\Sigma^{2p^{k-1}(p-1)}M(p^k)\to M(p^k).$$
We can use this map to
construct other $K(1)$-locally invertible spectra; in fact, we'll be able to
define an injection $\Z_p\to \pic(\Lone \Sp)$.

Let $a\in \Z_p$, so that $a = \sum^\infty_{k=0}\lambda_kp^k$.  Let $a_m$ denote
the truncation $\sum^m_{k=0}\lambda_kp^k$.  Define a spectrum $S^{-|v_1|a}$ by
the homotopy colimit of the diagram
$$\cdots\to \Sigma^{-|v_1|a_{k-1}}M(p^k)\to
\Sigma^{-|v_1|a_{k-1}}M(p^{k+1})\xar{v_1^{p^k\lambda_k}}\Sigma^{-|v_1|a_k}
M(p^{k+1})\to \Sigma^{-|v_1|a_k}M(p^{k+2})\to\cdots$$
If $a\in\Z\subset\Z_p$, then $\Lone S^{-|v_1|a} \simeq\Lone M(p^\infty)$, as
$\lambda_k = 0$ for $k\gg 0$.  Since $K(n)$-homology plays nicely with homotopy
colimits, we compute that
$$\dim_{K(1)_\ast}K(1)_\ast(S^{-|v_1|a}) = 1$$
for every $a\in\Z_p$.

This provides us with a continuous homomorphism $\Z_p\to \pi_0\pic\Lone \Sp$.
Hopkins-Mahowald-Sadofsky show that this is an injective homomorphism (we
will not, as this will take us too far afield), and the cosets of its image are
the ordinary spheres $S^1,\cdots,S^{|v_1|}$.  In particular, they construct a
short exact sequence
$$0\to \Z_p^\times\to \pi_0\pic\Lone\Sp \to \Z/2\to 0$$
and show that this does not split.  Since $\Z_p^\times\simeq \Z_p\times
\Z/(p-1)$, this implies that $\pi_0\pic\Lone\Sp\simeq \Z_p\times\Z/(2p-2)$. We
know that $|v_1| = 2(p-1)$, so the result follows.

\begin{proof}[Proof of Theorem \ref{hms}]
    Since $K(n)$ is a field spectrum, the implication (1) $\Rightarrow$ (2) is
    easy: if $M$ is $K(n)$-locally invertible, then there exists $M^\prime$
    such that $M\widehat{\wedge}M^\prime\simeq \Lk S$; the result follows by
    applying $K(n)$-homology and using the K\"unneth isomorphism.

    For the other direction, suppose $\dim_{K(n)_\ast}K(n)_\ast M = 1$.  Let $Z
    = \underline{\Map}(M,\Lk S)$; there is an evaluation map $M\wedge
    \underline{\Map}(M,\Lk S)\to \Lk S$. It suffices to show that this is an
    equivalence on $K(n)$-homology. Let $\cc$ be the subcategory of $\Sp$
    spanned by all spectra $X$ for which the map
    $$M\wedge \underline{\Map}(M,\Lk X)\xar{e_X} \Lk X$$
    is an equivalence on $K(n)$-homology. Any finite type $n$ spectrum $X$
    admits a finite filtration on $\Lk X$ with each cofiber a wedge of $K(n)$s.
    The category $\cc$ is closed under cofibrations and wedges, so to show that
    $e_X$ is an equivalence for any finite type $n$ spectrum, it suffices to
    observe that $e_{K(n)}$ is an equivalence on $K(n)$-homology.  Using the
    finiteness of $X$, we deduce that $e_X$ is an $K(n)$-equivalence if and
    only if $$M\wedge \underline{\Map}(M,\Lk S)\wedge X\to X\wedge \Lk S$$
    (which is the same map as $e_X$) is an equivalence. In turn, this happens
    if and only if $e_{S}$ is a $K(n)$-equivalence, as desired.

    Hopkins-Mahowald-Sadofsky prove that (2) is equivalent to (3).  We will
    instead show that (1) is equivalent to (3) using the tools from derived
    algebraic geometry developed in the previous sections. Lemma \ref{qcoh}
    shows that $\pic\Lk\Sp = \pic(\dX)$.  The Picard space satisfies
    descent\footnote{The Picard \emph{group}, however, generally does not
    satisfy any form of descent.}, and hence $\pic(\dX) \simeq \pic(\spf
    E)^{h\Gamma}$. Let $\tau:\spf E\to \dX$ denote the \'etale cover.

    Assume statement (1) of Theorem \ref{hms}, i.e., suppose $M$ is in
    $\pic\Lk\Sp$. Since $E$ satisfies the conditions appearing before
    Proposition \ref{picinv}, every invertible object of $\QCoh(\spf E)$ is of
    the form $\Sigma^k \mathcal{L}$ where $\mathcal{L}$ is a line bundle on
    $\spf E$ and $k\in\mathbf{Z}$. This means that we can assume that $\tau^\ast
    M$ is a line bundle.  It is not hard to prove that $\tau^\ast M \simeq E \
    \widehat{\wedge} \ M$. Since $\Gamma$ acts on the first factor, it follows
    that $E^\vee_\ast(M)$ is a free $E_\ast$-module of rank $1$.
    
    Now assume (3).  As a consequence of \cite{baker}, we know that $E \
    \widehat{\wedge} \ M$ is in $\pic(\spf E)$, where $M\in \QCoh(\dX)$.  It
    suffices to prove that this has a $\Gamma$-linearization.  But by
    Goerss-Hopkins-Miller $\Gamma$ acts continuously on $E \ \widehat{\wedge} \
    M$ via the first factor, and $E$ descends to the structure sheaf $\Lk S$ on
    $\dX$, so $E \ \widehat{\wedge} \ M$ has a $\Gamma$-linearization, as
    desired.
\end{proof}
\begin{remark}\label{piccyclic}
    The same argument proves that the following statements are equivalent, for
    $G$ a finite subgroup of $\Gamma$.
    \begin{itemize}
	\item An $E^{hG}$-module $M$ is in $\pic(E^{hG})$.
	\item $E^\vee_\ast M$ is a free $E^\vee_\ast E^{hG}$-module of rank $1$.
    \end{itemize}
\end{remark}
\begin{remark}
    A direct proof of the equivalence between (2) and (3) is also possible. By
    replacing $M$ be $\Sigma M$ if necessary, we may assume that $E^\vee_\ast
    M$ (resp. $K(n)_\ast M$) is concentrated in even degrees. Using
    \cite[Proposition 8.4]{hs}, we see that in this case, the rank of
    $E^\vee_\ast M$ as an $E_\ast$-module agrees with the dimension of
    $K(n)_\ast M$ as a $K(n)_\ast$-module. This is a version of Nakayama's
    lemma in the case of spectra with \emph{even} completed $E$-homology.
\end{remark}
\begin{lemma}
There is an equivalence $\pic(\spf E) \simeq \pic(E)$ that respects the
    $\Gamma$-action.
\end{lemma}
\begin{proof}
    This follows from \cite[Theorem 8.5.0.3]{SAG}.  Here is another, more
    topological, proof of this claim: clearly, any element of $\pic(E)$ is
    contained in $\pic(\spf E)$.  Conversely, an element of $\pic(\spf E)$ is
    contained in $\pic(E)$ if and only if it is a perfect $E$-module.  This
    follows from \cite[Proposition 10.11]{mathew-thesis}.
\end{proof}
This tells us that
$$\pic(E)^{hG} \simeq \pic(\spf E/G) \simeq \pic(E^{hG})$$
for any finite subgroup $G\subseteq\Gamma$.

Some results follow directly from our proof of Theorem \ref{hms}.
\begin{remark}
    There is a homotopy equivalence $\pic(E) \simeq BGL_1(E) \simeq
    \Omega^\infty \bgl_1 E$, which tells us that
    $$\pi_0\pic(E) \simeq \mathbf{Z}/2,$$
    and that
    $$\pi_1 \pic(E) \simeq
    (W(\mathbf{F}_{p^n})[[u_1,\cdots,u_{n-1}]])^\times.$$
    In fact, there is a fiber sequence
    $$\bgl_1(E) \to \mathfrak{pic}(E)\to H\mathbf{F}_2.$$
\end{remark}
The above results furnish a homotopy equivalence $\pic(E)^{h\Gamma} \simeq
\pic\Lk\Sp$, which gives a homotopy fixed points spectral sequence for
computing $\pic\Lk\Sp$ of signature
\begin{equation}\label{sseq}
E_2^{s,t} = \H^{s}_c(\Gamma;\pi_t \pic(E)) \Rightarrow \pi_{t-s} \pic\Lk\Sp.
\end{equation}
Note that $E_2^{1,1} \simeq \pic_n^{\mathrm{alg},0}$.

We remark that one can construct a map
$$ \epsilon:\pi_0 \pic \Lk \Sp\to \H^1_c(\Gamma; \pi_1 \pic E) \simeq
\H^1_c(\Gamma; E_0^\times) $$
as follows. Let $\cL$ be an element of $\pi_0 \pic \Lk \Sp$, thought of as (an
equivalence class of) a line bundle on $\spf E/\Gamma$. This is a
$\Gamma$-equivariant line bundle on $\spf E$. The underlying line bundle gives
rise to a $\Gamma$-equivariant line bundle on $\spf \pi_0 E$. The monodromy
action ($\Gamma$ is the ``\'etale fundamental group'' of the quotient stack
$\spf E/\Gamma$; see \cite{mathew-thesis}) gives rise to a (continuous)
representation
$$\Gamma\to \GL_1(\pi_0 E) = E_0^\times,$$
which gives the desired map $\epsilon$.

\begin{example}
    We can use the homotopy fixed point spectral sequence of Equation
    \eqref{sseq} to recover the result of Theorem \ref{piccomputation}.  First,
    suppose that $p$ is odd.  Recall (e.g., from \cite{henn}) that for $t>1$ we
    have
    $$E_2^{s,t} = \H^s_c(\Z_p^\times; \pi_t \pic(K_p)) \simeq \begin{cases}
	\Z/p^{\nu_p(t^\prime)+1} & t = 2(p-1)t^\prime+1, s=1\\
	0 & \text{else.}
    \end{cases}$$
    None of these terms contribute to the $t-s=0$ line.  Other contributions
    come from
    $$E_2^{1,1} = \H^1_c(\Z_p^\times,\Z_p^\times) \simeq \Z_p^\times \simeq
    \Z_p\times \Z/(p-1),$$
    and $E_2^{0,0} = \H^0_c(\Z_p^\times, \Z/2) \simeq \Z/2$.
    By sparseness, we learn that $E_2\simeq E_\infty$.  We are left with an
    extension problem on the line $t-s=0$, which is solved by \cite[Proposition
    2.7]{hmsref}.  If $p=2$, the same argument works, although in this case
    $E_2^{1,1} \simeq \Z_2^\times\times \Z/2$, and the extension problem is
    trivial.
\end{example}
Likewise, the equivalence
$$\pic(E)^{hG} \simeq \pic(E^{hG})$$
beget a spectral sequence
$$E_2^{s,t} = \H^s(G; \pi_t \pic(E)) \Rightarrow \pi_{t-s} \pic(E^{hG}).$$

\subsection{Anderson self-duality}
In this section, we will abuse notation by writing $I_\Z$ for $\Lk I_{\Z_p} \Lk
S$. If $G\subseteq \Gamma$ is a finite subgroup of the Morava stabilizer group
(and if $G = \Gamma$), the pushforward $q_\ast$ coming from the quotient map
$q:\spf E \to \spf E/G$ admits a right adjoint $q^!$. Explicitly, one has
$$q^!(M) = \Lk\underline{\Map}_{E^{hG}}(E, M).$$
We begin with the trivial observation that $\spf E$ is self-dual.
\begin{theorem}\label{genbbs}
    Let $G$ be a finite subgroup of $\Gamma$. Then $I_\Z E^{hG}$ is in the
    Picard group of $E^{hG}$.
\end{theorem}
\begin{proof}
    Let $G$ be a finite subgroup of $\Gamma$. Then there is an equivalence
    $$\QCoh(\spec E/G)\simeq \Mod(E)^{hG}.$$
    Since the extension $E^{hG}\to E$ is $G$-Galois (\cite[Example
    6.2]{mathew-meier}), there is an equivalence
    $$\Mod(E)^{hG}\simeq \Mod(E^{hG}).$$
    Utilizing Lemma \ref{etaledual}, we learn that $\spec E/G$ is self-dual,
    so that Theorem \ref{andinvert} (and Remark \ref{redemption}) shows that
    $I_\Z E^{hG}$ is in
    $$\pic \spec E/G \simeq \pic(E^{hG})\simeq \pic(E)^{hG}.$$
\end{proof}
In future work, we will generalize this (using Theorem \ref{andinvert} again)
to ``global'' cases like $\Tmf$ with level structure, and PEL Shimura varieties
as considered in \cite{BL}, as well as to genuine $K$-equivariant versions,
where $K$ is a finite abelian group.

As a corollary, we obtain a reproof of a consequence of a recent result of
Barthel-Beaudry-Stojanoska (\cite{barthelbeaudrystojanoska}).
\begin{corollary}\label{genbbs-cor}
    Let $G$ be a finite subgroup of $\Gamma$ at height $p-1$. Then $\Lk I_\Z
    E^{hG}$ is equivalent to a shift of $E^{hG}$.
\end{corollary}
\begin{proof}
    At height $p-1$, since $\pi_0 \pic(E^{hG})$ is cyclic (\cite{pichms}), we
    conclude from Theorem \ref{genbbs} that $E^{hG}$ is Anderson self-dual.
\end{proof}
\begin{remark}\label{sw-self-duality}
    We can deduce the $K(n)$-local Spanier-Whitehead self-duality of $E^{hG}$
    at height $p-1$ from the above example. (This self-duality is true more
    generally, as we will prove below, but this example illustrates an
    application of Theorem \ref{genbbs}.) Since $I_\Z$ is invertible by
    Gross-Hopkins duality (see Remark \ref{grosshopkins}), we know that
    $$DE^{hG} \simeq I_\Z^{-1}\widehat{\wedge} I_\Z E^{hG}.$$
    From the above example, we know that $\Lk I_\Z E^{hG}$ is equivalent to a
    shift of $E^{hG}$ at $n=p-1$. We will be done if $I_\Z \widehat{\wedge}
    E^{hG}$ is equivalent to a shift of $E^{hG}$. As
    $$(I_\Z \widehat{\wedge} E^{hG}) \widehat{\wedge}_{E^{hG}} M \simeq I_\Z
    \widehat{\wedge} M,$$
    we can use Gross-Hopkins duality to deduce that $M = E^{hG}
    \widehat{\wedge} I_\Z^{-1}$ is an inverse to $I_\Z \widehat{\wedge} E^{hG}$
    in $\Lk \Mod(E^{hG})$. It follows from $\pi_0 \pic(E^{hG})$ being cyclic
    that $I_\Z \widehat{\wedge} E^{hG}$ is a shift of $E^{hG}$, as desired.
\end{remark}
\begin{remark}
    For instance, we recover the well-known result that $KO^\wedge_2$ is
    $K(1)$-locally Spanier-Whitehead self-dual. At the prime $3$, there is an
    equivalence $L_{K(2)} TMF \simeq EO_2$; therefore, we also recover the
    $K(2)$-local Spanier-Whitehead self-duality of $L_{K(2)} TMF$. This result
    is originally due to Behrens (\cite[Proposition 2.6.1]{behrens}).
\end{remark}
This motivates a natural conjecture, which is widely believed to be true:
\begin{conjecture}\label{sw-dual}
    Let $G\subseteq \Gamma$ be a finite subgroup of the Morava stabilizer group
    at height $n$. Then $D(E^{hG}) \simeq (DE)^{hG}$ is a shift of $E^{hG}$,
    i.e., $E^{hG}$ is Spanier-Whitehead self-dual.
\end{conjecture}
\begin{remark}
    Conjecture \ref{sw-dual} is true if $(p-1)$ does not divide $n$. In
    Appendix \ref{appendix}, we prove Conjecture \ref{sw-dual} when $(p-1)$
    divides $n$ in the case when $G$ has Sylow $p$-subgroup $C_p$ (hinging on
    unpublished work of Hill-Hopkins-Ravenel in \cite{hhr} and \cite{hhr-eon}).
    This property is satisfied by all finite subgroups with nontrivial
    $p$-torsion of the Morava stabilizer group whenever $p$ does not divide
    $n/(p-1)$.
\end{remark}
\begin{definition}
    Let $\kappa(G)$ be the group of ``exotic'' invertible $E^{hG}$-modules,
    i.e., the group of invertible $E^{hG}$-modules $M$ such that, as
    $E_\ast[[\Gamma]]$-modules, $E^\vee_\ast(M) \simeq E^\vee_\ast(E^{hG})$.
\end{definition}
Conditional on Conjecture \ref{sw-dual}, we obtain the following result (whose
proof is just Remark \ref{sw-self-duality} run backwards), which is a
generalization of Corollary \ref{genbbs-cor}:
\begin{theorem}\label{andersonselfdual}
    Assume Conjecture \ref{sw-dual}. Suppose $G\subset \Gamma$ is a finite
    subgroup. If $\kappa(G)$ is cyclic or trivial, then $\Lk I_\Z E^{hG}$ is
    equivalent to a shift of $E^{hG}$.
\end{theorem}
\begin{proof}
    Since $I_\Z$ is $K(n)$-locally invertible (see Remark \ref{grosshopkins}),
    there is an equivalence
    $\Lk I_\Z E^{hG} \simeq I_\Z^{-1} \widehat{\wedge} DE^{hG}$.
    The Tate spectrum $E^{tG}$ is contractible, so
    $$DE^{hG} \simeq \underline{\Map}(E^{hG}, \Lk S) \simeq
    \underline{\Map}(E_{hG}, \Lk S) \simeq \underline{\Map}(E, \Lk S)^{hG}
    \simeq (DE)^{hG}.$$
    By Conjecture \ref{sw-dual}, $DE^{hG}$ is equivalent to a shift of
    $E^{hG}$. We are reduced to proving that $I_\Z^{-1}\widehat{\wedge} E^{hG}$
    is equivalent to a shift of $E^{hG}$.

    To prove that $I_\Z \widehat{\wedge} E^{hG}$ is equivalent to a shift of
    $E^{hG}$, we need to understand the image of $\pic_n$ inside
    $\pic(E^{hG})$, under the map $\pic_n\to \pic(E^{hG})$ given by $X\mapsto
    X\widehat{\wedge} E^{hG}$. Our hypotheses on $\kappa(G)$ are enough to
    guarantee that the image of $\pic_n$ inside $\pic(E^{hG})$ is cyclic; this
    shows that $I_\Z^{-1} \widehat{\wedge} E^{hG}$ is equivalent to a shift of
    $E^{hG}$, as desired.
\end{proof}
We illustrate some examples of Theorem \ref{andersonselfdual}.
\begin{remark}
    Suppose $G$ has order coprime to $p$. We claim that $\Lk I_\Z E^{hG} \simeq
    \Sigma^? E^{hG}$. This is the easiest case of Theorem
    \ref{andersonselfdual}, so we will provide two proofs.
    \begin{enumerate}
	\item It follows from the homotopy fixed point spectral sequence for
	    $\pic(E^{hG})$ that $\pi_0\pic(E^{hG})$ is cyclic if $\gcd(|G|, p)
	    = 1$. Since $I_\Z E^{hG}$ is an invertible $E^{hG}$-module, it
	    follows that $E^{hG}$ is Anderson self-dual.
	\item We claim that $\kappa(G) = 0$; the desired result follows from
	    Theorem \ref{andersonselfdual}. Let $X\in \kappa(G)$, and pick an
	    isomorphism $f:E^\vee_\ast(X) \xar{\sim} E^\vee_\ast E^{hG}$.
	    Shapiro's lemma provides an isomorphism $\wt{f}:\H^\ast(G;\pi_\ast
	    E) \xar{\sim} \H^\ast_c(\Gamma; E^\vee_\ast(X))$. Since
	    $\gcd(|G|,p) = 1$, the group cohomology $\H^s(G;\pi_\ast E)$ is
	    trivial for $s>0$.  Any differential $d_k^X:\H^0_c(\Gamma;
	    E^\vee_0(X)) \to \H^k_c(\Gamma; E^\vee_{k+1}(X))$ is therefore
	    trivial, so the identity class in $\H^0_c(\Gamma; E^\vee_0(X))$
	    survives to the $E_\infty$-page; this begets a map $\Lk S \to X$,
	    which extends to an equivalence $X \simeq E^{hG}$. Since $X$ was
	    arbitrary, we conclude that $\kappa(G) = 0$. If $n$ is not
	    divisible by $p-1$, it is known that all maximal finite subgroups
	    $G\subseteq \Gamma$ have order coprime to $p$ (this is proved, for
	    instance, in \cite[Theorem 1.3]{hewett} and \cite[Proposition
	    1.7]{bujard}). The above discussion now implies that $\Lk I_\Z
	    E^{hG}$ is equivalent to a shift of $E^{hG}$.
    \end{enumerate}
\end{remark}

\begin{example}
    At height $2$ and the prime $2$, it is known that if $G$ contains all the
    $p$-torsion in $\Gamma$, the group $\kappa(G)$ is isomorphic to $\Z/8$
    (\cite[Page 18]{beaudry-slides}). Theorem \ref{andersonselfdual} proves
    that at $p=2$, the spectrum $\Ltwo I_\Z E^{hG}$ is equivalent to a shift of
    $E^{hG}$.
\end{example}

\begin{remark}
    One does not need $\kappa(G)$ to vanish in order to get self-duality: if
    $F\subseteq \Gamma$ (at any height and prime) is in the kernel of the
    determinant map, then $\spf E/F$ is self-dual; indeed, the proof of
    Proposition \ref{andersonselfdual} shows that, in order to prove the
    Anderson self-duality of $E^{hF}$, we only need to know that $\Lk
    I_\Z^{-1}\wedge E^{hF}$ is equivalent to a shift of $E^{hF}$. This follows
    (e.g., from analyzing the homotopy fixed points spectral sequence) from the
    fact that $F\subseteq \ker\det$.
\end{remark}

In \cite{gross}, Gross and Hopkins prove the following result.

\begin{remark}[Gross-Hopkins duality]\label{grosshopkins}
    Let $MS$ denote the fiber of the map $L_n S\to L_{n-1} S$. Gross-Hopkins
    duality asserts that the spectrum $I_{\QQ/\Z} MS$ is invertible. There is
    an equivalence $$I_{\QQ/\Z} MS \simeq \Sigma \Lk I_\Z \Lk S.$$ This follows
    immediately from the fact that $\Lk I_{\QQ} X \simeq 0$. It therefore
    suffices to prove that $f^! I_\Z$ (whose underlying $K(n)$-local spectrum
    is $\Lk I_\Z \Lk S$) is invertible, where $f:\spf E/\Gamma\to \spec S$ is
    the structure map.

    We have an \'etale cover $\tau:\spf E\to \spf E/\Gamma$, but it is not a
    finite morphism. This map therefore does not satisfy the hypotheses of
    Lemma \ref{etaledual}. However, one can use the equivalence (see
    \cite{strickland}) $\Sigma^{n^2} DE \simeq E$ to show that $\spf E/\Gamma$
    is self-dual. In order to establish that $f^! I_\Z$ is invertible, it
    suffices to establish an analogue of Theorem \ref{andinvert}. However, the
    finiteness assumptions there do not apply to $\tau$, so we do not know how
    to prove this.
\end{remark}

\begin{remark}
    As $f^! I_\Z$ is invertible, the dualizing spectrum $\tau^\ast f^! I_\Z$
    defines a line bundle on $\spf E$. The action of $\Gamma$ defines a map
    $\Gamma\to \GL_1(E)$. Gross and Hopkins show that composing with the map
    $\GL_1(E)\to \GL_1(\pi_0 E)$ defines the determinant representation of
    $\Gamma$. We will return to the problem of proving this result via derived
    algebro-geometric methods in a later paper.
\end{remark}

\appendix

\section{Spanier-Whitehead self-duality of $E_{n(p-1)}^{hG}$}\label{appendix}
In this section, we will work at height $n(p-1)$ for some integer $n$. Fix the
notation $G$ for a finite subgroup of $\Gamma$ whose Sylow $p$-subgroup is
$C_p$. In this section, we will prove the following two results:
\begin{prop}\label{conj-proof}
    Under the assumptions in the beginning of this section, Conjecture
    \ref{sw-dual} is true for $E^{hG}$.
\end{prop}
\begin{remark}
    Note that every finite subgroup of $\Gamma$ with nontrivial $p$-torsion has
    Sylow $p$-subgroup $C_p$ whenever $p$ does not divide $n$, so Conjecture
    \ref{sw-dual} is true for \emph{every} finite subgroup $G$ in this case.
\end{remark}
Our proofs are computational. We will prove Proposition \ref{conj-proof} by
following the argument in \cite[Corollary 4.11]{barthelbeaudrystojanoska}. We
will rely on the following unpublished result of Hill-Hopkins-Ravenel from
\cite[Propositions 1 and 2]{hhr-eon} (see also \cite{hhr} for a more detailed
exposition in the case $n=2$):
\begin{theorem}[Hill-Hopkins-Ravenel]\label{eon}
    Modulo the image of the transfer (all such elements are permanent cycles),
    the $E_2$-term of the HFPSS for $E_{n(p-1)}^{hC_p}$ is given by
    \begin{align*}
	\Lambda(\alpha_1, \cdots, \alpha_n) \otimes
	P(\beta,\delta_1,\cdots,\delta_n^{\pm 1}),
    \end{align*}
    where the bidegrees of the elements, written in Adams indexing, are
    $|\alpha_i| = (-3,1)$, $|\beta| = (-2, 0)$, and $|\delta_i| = (-2p,0)$.
    Moreover, all of the differentials are determined by
    \begin{itemize}
	\item For $1\leq i\leq n$, there are differentials
	    \begin{equation}\label{important-diffl}
		d_{2p^i-1}(\delta_n^{p^{i-1}}) = a_i \delta_n^{p^{i-1}}
		h_{i,0} \beta^{p^i-1};
	    \end{equation}
	    here, $h_{i,0}$ are certain elements obtained by translating the
	    elements $\alpha_i$ by powers of $\delta_n$, and the elements $a_i$
	    are units in $\FF_{p^n}$.
	\item For $1\leq i\leq n$, there are ``Toda-style'' differentials on
	    the $E_{2(p^i-1)(p-1)-1}$-page which truncate the $\beta$-towers on
	    $\delta_i$.
	\item The classes $\delta_i \delta_n^{-1}$ and $\delta_n^{p^n}$ are
	    permanent cycles.
    \end{itemize}
\end{theorem}
\begin{proof}[Proof of Proposition \ref{conj-proof}]
    Let $E = E_{n(p-1)}$, and let $G = C_p$. According to \cite[Proposition
    16]{strickland}, the $\pi_\ast E^{hG}$-module $\pi_\ast DE^{hG}$ is free of
    rank one as a $C_p$-$\pi_\ast E$-module on a generator that we shall denote
    by $\gamma$, and the HFPSS for $DE^{hG}$ is a module over that of $E^{hG}$.
    These spectral sequences collapse at a finite stage, so by \cite[Lemma
    4.7]{barthelbeaudrystojanoska}, it suffices to prove that $\delta_n^N
    \gamma$ is a permanent cycle for some integer $N$.

    Before proceeding with the proof, let us show how this proves the result
    for finite subgroups $G\subsetneq \Gamma$ with Sylow $p$-subgroup $C_p$. As
    the Leray-Hochschild-Serre spectral sequence degenerates, there is an
    isomorphism of $E_2$-pages
    $$\H^\ast(G, \pi_\ast DE) \simeq \H^\ast(C_p, \pi_\ast DE)^{G/C_p}.$$
    The norm of $\delta_n^N \gamma$ under the action of $G/C_p$ is a permanent
    cycle in the HFPSS for $DE^{hG}$, so we are done.

    To prove the result when $G = C_p$, we argue inductively. It follows from
    Theorem \ref{eon} that $\gamma$ is a $(2p-2)$-cycle, and that
    $$d_{2p-1}(\gamma) = b_1 h_{1,0} \beta^{p-1} \gamma,$$
    for some unit $b_1\in \FF_{p^n}^\times$. It follows that
    $$d_{2p-1}(\delta_n^{k_1} \gamma) = (k_1 a_1 + b_1) h_{1,0} \beta^{p-1}
    \delta_n^{k_1-1} \gamma$$
    is zero if $k_1$ is chosen to be congruent to $-b_1/a_1$ modulo $p$.
    Therefore, $\delta_n^{k_1} \gamma$ is a $(2p-1)$-cycle (and hence a
    $(2p^2-2)$-cycle, by sparsity). For the inductive step, suppose
    $\delta_n^{k_i} \gamma$ is a $(2p^{i+1}-2)$-cycle; we need to show that
    there is some $N$ such that $\delta_n^N \gamma$ is a $(2p^{i+1}-1)$-cycle.
    We have
    $$d_{2p^{i+1}-1}(\delta_n^{k_i} \gamma) = b_{i+1} h_{i+1,0}
    \beta^{p^{i+1}-1} \gamma$$
    for some $b_{i+1}\in \FF_{p^n}^\times$. Arguing as above, we have
    $$d_{2p^{i+1}-1}(\delta_n^{\ell_{i+1}p^i+k_i}\gamma) = (\ell_{i+1}a_{i+1} +
    b_{i+1}) \delta_n^{k_{i+1}p^i} h_{i+1,0} \beta^{p^{i+1}-1} \gamma,$$
    so choosing $\ell_{i+1}$ congruent to $-b_{i+1}/a_{i+1}$ modulo $p$, we find
    that $\delta_n^{\ell_{i+1} p^i + k_i} \gamma$ is a $(2p^{i+1}-1)$-cycle, as
    desired. Having completed the inductive step, we find that $DE^{hC_p}$ is a
    shift of $E^{hC_p}$ by $2pN = 2p\sum_{i=0}^{n} \ell_{i+1} p^i$. This
    finishes the proof of Proposition \ref{conj-proof}.
\end{proof}

\bibliographystyle{alpha}
\bibliography{main}
\end{document}